\documentclass{amsart}
\usepackage{amsmath}
\usepackage{amssymb}
\usepackage{paralist}
\usepackage{tikz}
\usepackage{graphics} 
\usepackage{epsfig} 
\usepackage{graphicx}  \usepackage{epstopdf}
\usepackage[colorlinks=true]{hyperref}
\hypersetup{urlcolor=blue, citecolor=red}

\newtheorem{theorem}{Theorem}[section]
\newtheorem{corollary}{Corollary}
\newtheorem{lemma}[theorem]{Lemma}
\newtheorem{proposition}{Proposition}

\newtheorem{definition}[theorem]{Definition}
\newtheorem{remark}{Remark}

\usepackage[mathscr]{euscript}
\newcommand{\R}{{\mathbb R}}
\newcommand{\T}{{\mathbb T}}                   
\newcommand{\N}{{\mathbb N}}                  
\newcommand{\C}{{\mathbb C}}                  
\newcommand{\F}{{\mathbb F}}
\newcommand{\D}{{\mathbb D}}

\newcommand{\Ss}{{\mathbb S}}

\def\ka{\kappa}
\def\al{\alpha}
\def\om{\omega}

\def\ga{\gamma}

\def\Si{\Sigma}
\def\la{\lambda}

\def\t{\tau}

\def\calR{{\mathcal{R}}}
\def\calL{{\mathcal{L}}}

\def\calA{{\mathcal{A}}}

\def\calC{{\mathcal{C}}}

\def\calD{{\mathcal{D}}}

\def\calP{{\mathcal{P}}}
\def\R{\mathbb R}

\def\P{\mathbb P}
\def\D{\mathbb D}

\def\C{\mathbb C}

\def\N{\mathbb N}

\def\F{\mathbb F}

\def\T{\mathbb T}
\def\G{\mathbb G}
\begin{document}







\title[Maximal $L^p$-regularity]{Staffans-Weiss perturbations for Maximal $L^p$-regularity in Banach spaces}


\author[A. Amansag]{Ahmed Amansag}
\address{Department of Mathematics, Faculty of Sciences, Ibn Zohr University, Hay Dakhla, BP8106, 80000--Agadir, Morocco}
\author[H. Bounit]{Hamid Bounit}

\author[A. Driouich]{Abderrahim Driouich}

\author[S. Hadd]{Said Hadd}

\maketitle

\begin{abstract}
  In this paper we show that the concept of maximal $L^p$-regularity
is stable under a large class of unbounded perturbations, namely Staffans-Weiss perturbations. To that purpose, we first prove that the analyticity of semigroups is preserved under this class of perturbations, which is a necessary condition for the maximal regularity. In UMD spaces, $\calR$-boundedness conditions are exploited to give conditions guaranteing the maximal regularity. For non-reflexive Banach space,  a condition is imposed to the Dirichlet operator associated  to the boundary value problem to prove the maximal regularity. A Pde example illustrating the theory and an application to a class of non-autonomous perturbed boundary value problems are presented.
\end{abstract}





\section{Introduction}\label{sec:intro}
	
	In this paper we investigate the maximal $L^p$--regularity of evolution equations of the type
	\begin{equation}\label{PM} \begin{cases} \dot{z}(t)=A_m z(t)+Pz(t)+f(t),& t\ge 0,\cr z(0)=0,\cr Gz(t)=Kz(t),& t\ge 0,\end{cases}\end{equation}
	where $A_m:Z\subset X\to X$ is a linear closed operator in a Banach space $X$ with domain $D(A_m)=Z,$ $P:Z\to X$ is an additive linear perturbation of $A_m,$  $G,K:Z\to U$ are linear boundary operators  ($U$ is another Banach space) and $f\in L^p(\R^+,X)$ with $p\ge 1$ is a real number. Actually, we assume that $A:=A_m$ with domain $D(A)=\ker(G)$ is a generator of a strongly continuous semigroup $\T:=(\T(t))_{t\ge 0}$ on $X$.

	The concept of maximal regularity has been the subject of several works for many years, e.g.  \cite{CanVes,Dore,CoulDuon,Simon}, and the monograph  \cite{DenHiePru}. The main purpose of these works is to give sufficient conditions on the operator $A$ so as the problem \eqref{PM} with $P\equiv 0$ and $K\equiv 0,$ which can be written as  \begin{equation}\label{A-only}\dot{z}=Az+f,\;z(0)=0,\end{equation} has a maximal $L^p$--regularity. A necessary condition for the  maximal regularity of the evolution equation \eqref{A-only} is that $A$ is a generator of an analytic semigroup. This condition is also sufficient only if $X$ is a Hilbert space, see \cite{Simon}. In the case of UMD Banach space $X,$ Weis \cite{LutzWeis2} introduced necessary and sufficient conditions in terms of $\calR$-boundedness of the operator $A$ for the maximal $L^p$-regularity of \eqref{A-only}.
	
	Return now to our initial boundary problem \eqref{PM}. This later can be reformulated as
	\begin{equation}\label{calA-max} \dot{z}=(\calA+P)z+f,\qquad z(0)=0,\end{equation}
	where $\calA: D(\calA)\subset X\to X$ is the linear operator defined by
	\begin{equation*} \calA:=A_m,\qquad D(\calA)=\{x\in Z: Gx=Kx\}.\end{equation*}
	To study maximal regularity of the problem \eqref{calA-max}, we first study maximal regularity of the problem \begin{equation}\label{calA-only}\dot{z}=\calA z+f,\quad z(0)=0.\end{equation}
	
	In addition to our assumption at the beginning  of this section, we also suppose that $G:Z\to U$ is surjective. Let then $\D_\la\in\calL(U,Z)$ ($\la\in\rho(A)$) be the Dirichlet operator associated with $A_m$ and $G$, see Section \ref{sec:3}. In order to state our main results on well-posedness and maximal $L^p$-regularity of the problem \eqref{calA-max}, we select $B:=(\la I-A_{-1})\D_\la\in\calL(U,X_{-1})$ for $\la\in\rho(A)$, where $X_{-1}$ is the extrapolation space associated to $A$ and $X$, and $A_{-1}:X\to X_{-1}$ is the extension of $A$ to $X,$ which is a generator of a strongly continuous semigroup on $X_{-1}$. We assume that $B$ is a $p$-admissible control operator for $A$, see the next section for the definition and notation. The case where $K$ is bounded, i.e. $K\in\calL(X,U)$, is studied in the recent paper \cite{AmBoDrHa}. In the present paper,we assume that the boundary operator $K$ is unbounded, that is $K:Z\varsubsetneq X\to U$. This situation is quite difficult which needs additional assumption to treat the well-posedness and  maximal $L^p$-regularity. According to \cite{HaddManzoRhandi}, if we assume that $(A,B,K_{|D(A)})$ is regular on $X,U,U$ with $I_U:U\to U$ as an admissible feedback, then the problem \eqref{calA-only} is well-posed on the Banach space $X$. Moreover, if the problem \eqref{A-only} has a  maximal $L^p$-regularity and $\|\la \D_\la\|\le \kappa$ for any $\la>\la_0,$ where $\la_0\in\R$ and $\kappa>0$ are constants, then the problem \eqref{calA-only} has also the maximal $L^p$-regularity on a non reflexive Banach space $X$, see Theorem \ref{theorem:Staffans_Weiss_Non_Reflexive_Case}. On the other hand, we assume that  $(A,B,P_{|D(A)})$ generates a regular linear system on $X,U,X$. Then, in Theorem \ref{theorem:generalisation_of_Hadd_Manzo_Rhandi}, we prove that the problem \eqref{calA-max} is well-posed. Corollary \ref{SF-big-equa-result} shows that the problem \eqref{calA-max} has the  maximal $L^p$-regularity. If $X$ is a UMD space then we use $\calR$-boundedness to prove the  maximal $L^p$-regularity for the evolution equation \eqref{calA-max}, see Theorem \ref{theorem:staffans-weiss-R-boundedness} and Corollary \ref{SF-big-equa-result}. We mention that in \cite{HaHaKu}, the authors proved perturbation theorems for sectoriality and $\calR$-sectoriality in general Banach spaces. They gives conditions on intermediate spaces $Z$ and $W$ such that, for an operator $S:Z \rightarrow W$ of small norm, the operator $A+S$ is sectorial (resp. $\calR$-sectorial) provided $A$ is sectorial (resp. $\calR$-sectorial). Their results are obtained by factorizing $S=BC$. As $\calR$-sectoriality implies maximal regularity in UMD spaces, these theorems yield to maximal regularity perturbation only in UMD spaces.
	
	In the next section, we first recall the necessary material about feedback theory of infinite dimensional linear systems. We then use this theory to prove the well-posedness of the evolution equation \eqref{PM} in Section \ref{sec:3}. Our main results on maximal $L^p$--regularity for the problem \eqref{PM} are gathered in Section \ref{sec:umd} and Section \ref{sec:nonreflexive}. The last section is devoted to apply the obtained results to a class of non-autonomous boundary problems.
	
	{\bf Notation.} Hereafter $p,q\in [1,\infty]$ and $T>0$ are real numbers such that $\frac{1}{p}+\frac{1}{q}=1$. If $X$ is a Banach space, we denote by $L^p([0,T];X)$ the space of all $X$-valued Bochner integrable functions. For any $\theta \in (0,\pi)$,  $\Sigma_\theta$ is the following sector:
	$$
	\Sigma_\theta:=\{z\in \mathbb{C}^*; \vert arg z\vert < \theta \}.
	$$
	For any $\alpha \in \mathbb{R}$, the right half-plane is defined by
	$$
	\mathbb{C}_\alpha=:\{z\in \mathbb{C}; Re z>\alpha\}.
	$$
	Given a semigroup $\T:=(\T(t))_{t\geq 0}$ generated by an operator $A:D(A)\subset X\to X$, we will always denote by $\omega_0(\mathbb{T})$(or $\omega_0(A)$) the growth bound of this semigroup. The resolvent set of $A$ is denoted by $\rho(A)$. Preferably, we denote the resolvent operator of $A$ by $R(\la,A):=(\la-A)^{-1}$ for any $\la\in\rho(A),$ where the notation $\la-A$ means $\la I-A$.


\section{Feedback theory of infinite dimensional linear systems}\label{sec:2}
	In this section, we gather definitions and results from feedback theory of infinite dimensional linear systems mainly developed in the references \cite{Sala,Staff,TucWei,WeiRegu}.  We also give some new development of this theory. Hereafter, $X$ and $U$ are Banach spaces and $p\in [1,\infty[$.
	
	It is known (see e.g. \cite{Sala,Staff}) that partial differential equations with boundary control and point observation can be reformulated as the following distributed linear system
	\begin{equation}
	\label{linContSys}
	\begin{cases}
	\dot{x}(t) = Ax(t)+Bu(t) , & \mbox{ } t\geq 0 \\
	y(t) = Kx(t), & \mbox{}\\
	x(0)=x_0, & \mbox{}
	\
	\end{cases}
	\end{equation}
	where $A:D(A)\subset Z\subset X\to X$ is the generator of a strongly continuous semigroup $\T:=(\T(t))_{t\ge 0}$ on $X$ with $Z$ is a Banach space continuously and densely embedded in $X$, $B\in\calL(U,X_{-1})$ is a control operator such that
	\begin{equation*}
	R(\la,A_{-1})B\in\calL(U,Z),\qquad \la\in\rho(A),
	\end{equation*}
	and $K\in\calL(Z,U)$ is an observation operator. Here  $X_{-1}$ is the completion of $X$ with respect to the norm $\Vert R(\lambda,A)\cdot\Vert$. We recall that we can extend $\T$ to another strongly continuous semigroup $\T_{-1}:=(\T_{-1}(t))_{t\geq 0}$ on $X_{-1}$ with  generator $A_{-1}:X\to X_{-1}$, the extension of $A$ to $X$ (see e.g. \cite[chap.2]{EngNag}).
	The mild solution of the system \eqref{linContSys} is given by:
	\begin{equation}
	\label{solVarConstForm}
	x(t) = \T(t)x_0 + \int_0^t \T_{-1}(t-s)Bu(s) ds \qquad x_0 \in X,
	\end{equation}
	where the integral is taken in $X_{-1}$. Formally, the well-posedness of the system  \eqref{linContSys} means that the state satisfies $x(t)\in X$ for any $t\ge 0,$ the observation function $y$ is extended to a locally $p$-integrable function  $y\in L^p_{loc}([0,\infty),U)$ satisfying the following property:  for any $\t>0,$ there exists a constant $c_\tau>0$ such that
	\begin{align*}
	\|y\|_{L^p([0,\t],U)}\le c_\t \left(\|x_0\|+\|u\|_{L^p([0,\t],U)}\right),
	\end{align*}
	for any initial state $x_0\in X$ and any control function $u\in L^p_{loc}([0,\infty),U)$. In order to mathematically explain this concept, let us define
	\begin{equation*}
	C:=Ki\in\calL(D(A),U).
	\end{equation*}
	where $i$ is the canonical injection from $D(A)$ to $Z$.\\\\
	We also need the following definition.
	\begin{definition}
		\begin{itemize}
			\item[(i)]  $B\in \mathcal{L}(U,X_{-1})$ is called $p$-admissible control operator for $A$, if there exists $t_0 > 0$ such that :
			\begin{equation*}
			\Phi_{t_0} u := \int_0^{t_0} \T_{-1}(t-s)Bu(s) ds \in X
			\end{equation*}
			for any $u\in L^p_{loc}([0,\infty),U)$. We also say that $(A,B)$ is $p$-admissible.
			\item[(ii)] $C\in \mathcal{L}(D(A),Y)$ is called $p$-admissible observation operator for $A$, if there exist $\alpha >0$ and $\kappa:=\kappa_\al>0$ such that:
			\begin{equation}\label{walid}
			\int_0^\alpha \Vert C\T(t)x\Vert_Y^p dt \leq \ka^p \Vert x \Vert^p,
			\end{equation}
			for all $x\in D(A)$. We also say that $(C,A)$ is $p$-admissible.
		\end{itemize}
	\end{definition}
	Let us now describe some consequences of this definition. If $B$ is $p$-admissible control operator for $A,$ then by the closed graph theorem one can see that for any $t\ge 0,$ $$\Phi_t\in\calL(L^p([0,t],U),X).$$
	This implies that the state of the system  \eqref{linContSys} satisfies $x(t)=\T(t)x_0+\Phi_t u\in X$ for any $t\ge 0,\;x_0\in X$ and $u\in L^p_{loc}([0,\infty),U)$.  According to \cite{Weis-ad-cont}, for all $0<\t_1\le \t_2,$
	\begin{equation}\label{norm-phi-estim}
	\|\Phi_{\t_1}\|\le \|\Phi_{\t_2}\|.
	\end{equation} Now if $C$ is $p$-admissible observation operator for $A,$ then due to \eqref{walid}, the map $\Psi_\infty: D(A)\to L^p_{loc}([0,\infty),U)$ defined by $\Psi_\infty x := C\T(\cdot)x$, can be  extended to a bounded operator $\Psi_\infty: X\rightarrow L_{loc}^p([0,\infty);U)$. For any $x\in X$ and $t\ge 0$, we define the family $\Psi_t x := \Psi_\infty x$ on $[0,t]$. Then for all $t\ge 0,$ $$\Psi_t\in \mathcal{L}\left( X, L^p_{loc}([0,\infty),U)\right).$$
	On the other hand, let us consider the linear operator
	\begin{align*}
	D(C_\Lambda)&:=\left\{x\in X: \lim_{s \rightarrow +\infty} s C R(s,A)x \;\text{exists in}\; U\right\}\\
	C_\Lambda x &:= \lim_{s \rightarrow +\infty}s C R(s,A)x.
	\end{align*}
	Clearly, $D(A)\subset D(C_\Lambda)$ and $C_\Lambda=C$ on $D(A)$. This shows that $C_\Lambda$ is in fact an extension of $C,$ called the \textit{Yosida} extension of $C$ w.r.t. $A$. We note that if $C$ is $p$-admissible for $A,$ then $Range(\T(t))\subset D(C_\Lambda)$  and $
	(\Psi_\infty x)(t)=C_\Lambda \T(t)x$
	for any $x\in X$ and a.e. $t\ge0$.
	
	In the sequel, we assume that $B$ and $C$ are $p$-admissible for $A$ and set
	\begin{equation*}
	W^{2,p}_{0,loc}([0,\infty),U):=\left\{u\in W^{2,p}_{loc}([0,\infty),U):u(0)=0\right\}.
	\end{equation*}
	This space is dense in $L^p_{loc}([0,\infty),U)$. Remark that for any $u\in W^{2,p}_{0,loc}([0,\infty),U),\; t\ge 0$ and by assuming $0\in \rho(A)$ (without loss of generality) and using an integration by parts, we have
	\begin{equation*}
	\Phi_t u=R(0,A_{-1})Bu(t)-R(0,A)\Phi_t \dot{u}\in Z.
	\end{equation*}
	On the other hand, using the fact that $KR(0,A_{-1})B\in\calL(U),\;CR(0,A)\in\calL(X,U)$ and \eqref{norm-phi-estim}, the application $(t\mapsto K\Phi_t u)\in L^p_{loc}([0,\infty),U)$ for any $u\in W^{2,p}_{0,loc}([0,\infty),U)$. Thus we have defined an application
	\begin{equation}\label{inp-out-infini}
	\F_\infty: W^{2,p}_{0,loc}([0,\infty),U)\to L^p_{loc}([0,\infty),U),\qquad u\mapsto \F_\infty u= K\Phi_{\cdot} u.
	\end{equation}
	\begin{definition}\cite{WeiAdmiss}
		Let $B$ and $C$ be $p$-admissible control and observation operators for $A,$ respectively. We say that the triple $(A,B,C)$ generates a well-posed linear system on $X,U,U$, if the operator $\F_\infty$ defined by \eqref{inp-out-infini} satisfies the following property: For any $\al>0$ there exists a constant $\vartheta_\al>0$ such that for all $u\in  W^{2,p}_{0,loc}([0,\infty),U),$
		\begin{equation}\label{F-infty-estim}
		\left\| \F_\infty u\right\|_{L^p([0,\al],U)}\le \vartheta_\al \|u\|_{L^p([0,\al],U)}.
		\end{equation}
		The operator $\F_\infty$ is called the extended input-output operator associated to $(A,B,C)$.
	\end{definition}
	If the triple $(A,B,C)$ generates a well-posed linear system on $X,U,U$, then we have two folds: first the state of \eqref{linContSys} satisfies $x(t)\in X$ for all $t\ge 0,$ and second $\F_\infty$ has an extension $\F_\infty\in \calL(L^p_{loc}([0,\infty),U)),$ due to \eqref{F-infty-estim}. Observe that the observation function $y$ verifies
	\begin{equation}\label{y-localy}
	y(\cdot):=y(\cdot;x_0,u)=CT(\cdot)x_0+\F_\infty u=(\begin{smallmatrix}\Psi_\infty& \F_\infty\end{smallmatrix})(\begin{smallmatrix}x_0\\ u\end{smallmatrix}),
	\end{equation}
	for all $x_0\in D(A)$ and $u\in  W^{2,p}_{0,loc}([0,\infty),U)$. By   density of $D(A)\times W^{2,p}_{0,loc}([0,\infty),U)$ in $X\times L^p_{loc}([0,\infty),U)$, the function $y$ is extended to a function $y\in L^p_{loc}([0,\infty),U)$ such that
	\begin{equation*}
	y=\Psi_\infty x_0+\F_\infty u,\qquad \forall (x_0,u)\in X\times L^p_{loc}([0,\infty),U).
	\end{equation*}
	We now turn out to give a representation of the observation function $y$ in terms of the observation operator $C$ and the state $x(\cdot)$. To that purpose Weiss \cite{WeiRegu,WeiTrans} introduced the following subclass of well-posed linear systems.
	\begin{definition}
		Let the triple $(A,B,C)$ generates a well-posed linear system on $X,U,U$ with extended input-output operator $\F_\infty$. Thi system is called regular (with feedthrough $D=0$) if :
		$$
		\lim_{\tau \rightarrow 0} \frac{1}{\tau} \int_0^\tau (\mathbb{F}_\infty u_{z_0})(s)ds = 0
		,$$
		with $u_{z_0}(s) = z_0$ for all $s\geq 0 $, is a constant control function.
	\end{definition}
	According to Weiss \cite{WeiRegu,WeiTrans}, if $(A,B,C)$ generates a regular system on $X,U,U$, then the state and the observation function of the linear system \eqref{linContSys} satisfy
	\begin{equation}\label{representation-y}
	x(t)\in D(C_\Lambda)\quad\text{and}\quad y(t)=C_\Lambda x(t),
	\end{equation}
	for any initial state $x(0)=x_0\in X$, any control function $u\in L^p([0,\infty),U)$ and a.e. $t\ge 0$.
	\begin{definition}\label{Feedback-admisible}
		Let a triple  $(A,B,C)$ generates a well-posed system on $X,U,U$ with   extended input-output operator $\F_\infty$. Define
		$$
		\mathbb{F}_\tau u:= \mathbb{F}_\infty u \,\qquad\text{on}\quad [0,\tau].
		$$
		The identity operator $I_U:U\to U$ is called an admissible feedback for $\Sigma$ if the operator $I-\mathbb{F}_{t_0}:L^p([0,t_0],U)\to L^p([0,t_0],U)$ admits a (uniformly) bounded inverse for some $t_0>0$ (hence all $t_0>0$).
	\end{definition}
	A consequence of Definition \ref{Feedback-admisible} is that the feedback law $u=y(\cdot;x_0,u)$ has a sense. In fact, due to \eqref{y-localy} this is equivalent to $(I-\F_{\t})u=\Psi_\t x_0$ on $[0,\t]$. As $I-\F_{\t}$ is invertible in $L^p([0,\t],U),$ then the equation $u=y(\cdot;x_0,u)$ has a unique solution and this solution $u\in L^p([0,\t],U)$ is given also by
	\begin{equation*}
	u(t)=C_\Lambda x(t),\qquad a.e.\; t\ge 0,
	\end{equation*}
	due to \eqref{representation-y}. Using \eqref{solVarConstForm}, the state $x(\cdot)$ satisfies the following variation of constants formula
	\begin{equation*}
	x(t)=\T(t)x_0+\int^t_0 \T_{-1}(t-s)BC_\Lambda x(s)ds,
	\end{equation*}
	for any $x_0\in X$ and any $t\ge 0$. Now put set
	\begin{equation*}
	\T^{cl}(t)x_0:=x(t),\qquad t\ge 0.
	\end{equation*}
	Then by using the definition of $C_0$--semigroups one can see that $(\T^{cl}(t))_{t\ge 0}$ is a $C_0$--semigroup on $X$. More precisely, we have the following perturbation theorem due to Weiss \cite{WeiRegu} in Hilbert spaces and to Staffans \cite[Chap.7]{Staff} in Banach spaces.
	\begin{theorem}\label{staff_wei}
		Let $(A,B,C)$ generates a regular linear system with the identity operator  $I_U:U\to U$ an admissible feedback operator. Then the operator
		\begin{equation} \label{A-cl}
		\begin{split}
		A^{cl} &:= A_{-1} + BC_\Lambda\\
		D(A^{cl}) &:= \{x\in D(C_\Lambda); (A_{-1} + BC_\Lambda)x\in X\}
		\end{split}
		\end{equation}
		generates a $C_0$-semigroup $(\T^{cl}(t))_{t\geq 0}$ on $X$ such that $Range(\T^{cl}(t))\subset D(C_\Lambda)$ for a.e. $t>0$, and  for any $\al>0,$ there exists $c_\al>0$ such that for all $x_0\in X,$
		\begin{equation}\label{T-cl-estim}
		\|C_\Lambda \T^{cl}(\cdot)x_0\|_{L^p([0,\al],U)}\le c_\al \|x_0\|.
		\end{equation}
		Moreover, this semigroup satisfies
		\begin{equation}\label{WS-VCF}
		\T^{cl}(t)x_0 = \T(t)x_0 + \int_0^t \T_{-1}(t-s)BC_{\Lambda}\T^{cl}(s)x_0 ds \qquad \qquad x_0\in X\ ,\ t \geq 0.
		\end{equation}
		In addition  $(A^{cl},B,C_\Lambda)$ generates a regular system.
	\end{theorem}
	
	\begin{definition}\label{Staffans-Weiss-perturbation-oper}
		Let $(A,B,C)$ generates a regular linear system on $X,U,U$ with the identity operator  $I_U:U\to U$ as an admissible feedback. The operator $$\P^{sw}:=BC:D(C_\Lambda)\subset X\to X_{-1}$$ is called the {\bf Staffans-Weiss perturbation} of $A$.
	\end{definition}
	
	It is not difficult to see that if one of the operators $B$ or $C$ is bounded (i.e. $B\in\calL(U,X)$ or $C\in\calL(X,U)$) and the other is $p$-admissible then the triple $(A,B,C)$ generates a regular linear system on $X,U,U$ with the identity operator $I_U:U\to U$ as an admissible feedback. As application of the Staffans-Weiss theorem (Theorem \ref{staff_wei}), we distinct two subclasses of perturbations as follows:
	\begin{remark}\label{Hadd_thm}
		\begin{itemize}
			\item [{\rm (i)}] We take $B\in \calL(X,U)$ and $C\in\calL(D(A),U)$ a $p$-admissible observation operator for $A$. According to Theorem \ref{staff_wei}, the operator $A^{cl}:=A+BC$ with domain $D(A^{cl})=D(A)$ is a generator of a strongly continuous semigroup $\T^{cl}:=(\T^{cl}(t))_{t\ge 0}$ on $X$ such that $\T^{cl}(t)X\subset D(C_\Lambda)$ for a.e. $t>0,$ the estimate \eqref{T-cl-estim} holds, and
			\begin{equation}\label{MV-VCF1}
			\T^{cl}(t)x=\T(t)x+\int^t_0 \T(t-s)BC_\Lambda \T^{cl}(s)xds,\qquad t\ge 0,\; x\in X.
			\end{equation}
			On the other hand, it is shown in \cite{Hadd}, that the semigroup $\T^{cl}$ satisfy also the following formula
			\begin{equation}\label{MV-VCF2}
			\T^{cl}(t)x=\T(t)x+\int^t_0 \T^{cl}(t-s)BC_\Lambda \T(s)xds,\qquad t\ge 0,\; x\in X.
			\end{equation}
			Using H\"{o}lder inequality one can see that there exist $\al_0>0$ and $\ga\in (0,1)$ such that
			\begin{equation*}
			\int^{\al_0}_0 \|BC\T(t)x\|dt\le \ga \|x\|,
			\end{equation*}
			for all $x\in D(A)$.\\ The following operator
			$$\P^{mv}:=BC: D(A)\to X $$
			is called a {\bf Miyadera-Voigt perturbation} for $A;$ (see e.g. \cite[p.195]{EngNag}).
			\item [{\rm (ii)}] We take $C\in\calL(X,U)$ and $B\in\calL(U,X_{-1})$ a $p$-admissible control operator for $A$. Then the part of the operator $A_{-1}+BC$ in $X$ generates a strongly continuous semigroup on $X$ satisfying all properties of Theorem \ref{staff_wei}. In this case the operator $$\P^{ds}:=BC:X\to X_{-1}$$ is called {\bf Desch-Schappacher perturbation} for $A$ (see e.g. \cite[p.182]{EngNag}).
		\end{itemize}
	\end{remark}


\section{Well-posedness of perturbed boundary value problems}\label{sec:3}
	The object of this section is to investigate the well-posedness of the perturbed boundary value problem defined by \eqref{PM}. We first rewrite \eqref{PM} as non-homogeneous perturbed Cauchy problem of the form \eqref{calA-max}. Then the  well-posedness of \eqref{PM} can be obtained if for example the operator
	\begin{equation}\label{cald-def} \calA:=A_m,\qquad D(\calA)=\{x\in Z: Gx=Kx\},\end{equation}  generates a strongly continuous semigroup on $X$ and that $P$ is a $p$-admissible observation operator for $A$ (see Remark \ref{Hadd_thm} (i)). Recently, the authors of \cite{HaddManzoRhandi} introduced conditions on $A_m,G$ and $K$ for which $\calA$ is a generator. To be more precise, assume that
	\begin{itemize}
		\item[] \textbf{(H1)} $G:Z\to U$ is onto,
		\item[] \textbf{(H2)} The operator defined by $A:=A_{m}$ and $D(A):=\ker(G)$, generates a $C_0$-semigroup $(\T(t))_{t\geq 0}$ on $X$.
	\end{itemize}
	According to  Greiner \cite{Grei}, these conditions imply that for any $\la\in\rho(A)$ the restriction of $G$ to $\ker(\la-A_m)$ is invertible. We then define
	\begin{equation*}
	\D_\la:=(G _{|Ker(\lambda-A_m}))^{-1}\in\calL(U,X),\qquad \la\in\rho(A).
	\end{equation*}
	This operator is called the {\em Dirichlet operator}.  Define the operators :
	\begin{align}\label{B_op}
	\begin{split}
	B:= & (\lambda - A_{-1})\mathbb{D}_\lambda \in \mathcal{L}(U,X_{-1}), \\
	C:=  & Ki \in \mathcal{L}(D(A),U),
	\end{split}
	\end{align}
	where $i$ is the canonical injection from $D(A)$ to $Z$.\\ Using the resolvent equation, and the fact that for all $\lambda,\mu\in \rho(A) $ (see \cite[Lemma 1.3]{Grei})
	\begin{equation}\label{GreiRela}
	\mathbb D_\lambda=(I-(\lambda -\mu)R(\mu, A))\mathbb D_\mu,
	\end{equation}
	it is easy to see that the operator $B$ does not depend on $\lambda$ .\\\\
	 In the rest of this paper, $C_\Lambda$ denotes the Yosida extension of $C$ with respect to $A$.
	
	It is shown in \cite[Lemma.3.6]{HaddManzoRhandi} that if $A,B,C$ as above and if $(A,B,C)$ generates a regular linear system on $X,U,U,$ then we have
	\begin{equation}\label{relation-K-C-Lambda}
	Z\subset D(C_\Lambda)\quad\text{and}\quad (C_\Lambda)_{|Z}=K.
	\end{equation}
	If $H$ is the transfer function of this system and $\al>\om_0(A)$ then
	\begin{equation}\label{transfer-function-ABC}
	H(\la)=C_\Lambda R(\la,A_{-1})B=C_\Lambda \D_\la=K\D_\lambda,
	\end{equation}
	for any $\la\in\C$ with ${\rm Re}\la >\al$. Moreover, we have
	\begin{equation}\label{limit-trans-func}
	\lim_{s\to+\infty}H(s)=0.
	\end{equation}
	We also assume the following
	\begin{itemize}
		\item[] \textbf{(H3)} the triple $(A,B,C)$ generates a regular linear system on $X,U,U$ with \qquad $I_U:U\to  U$ as an admissible feedback operator.
	\end{itemize}
We have the following perturbation theorem (see \cite{HaddManzoRhandi} for the proof).
	\begin{theorem}\label{theorem:Hadd_Manzo_Rhandi}
		Let assumptions {\rm \textbf{(H1)}} to {\rm \textbf{(H3)}} be satisfied. Then, the following assertions hold:
		\begin{itemize}
			\item [{\rm (i)}]The operator $(\calA,D(\calA))$ defined by \eqref{cald-def} coincides with the following operator
			\begin{equation*}
			A^{cl}:=A_{-1}+BC_\Lambda,\quad D(A^{cl})=\{x\in D(C_\Lambda):(A_{-1}+BC_\Lambda)x\in X\}.
			\end{equation*}
			\item [{\rm (ii)}] The operator $(\calA,D(\calA))$ generates a strongly continuous semigroup $(\T^{cl}(t))_{t\geq 0}$ on $X$ as in Theorem \ref{staff_wei}.
			\item [{\rm (iii)}] For any $\lambda\in\rho(A)$ we have
			\begin{equation*}
			\lambda\in\rho(\calA) \Leftrightarrow 1 \in \rho(\mathbb{D}_\lambda K) \Leftrightarrow 1 \in \rho(K\mathbb{D}_\lambda).
			\end{equation*}
			\item [{\rm (iv)}] Finally for $\lambda \in \rho(A)\cap \rho(\calA)$:
			\begin{equation*}
			R(\lambda,\calA) = (I-\mathbb{D}_\lambda K)^{-1}R(\lambda,A).
			\end{equation*}
		\end{itemize}
	\end{theorem}
	Under the assumptions of Theorem \ref{theorem:Hadd_Manzo_Rhandi}, the mild solution of the problem \eqref{calA-only} is given by
	\begin{equation}\label{formula-used}
	z(t)=\T^{cl}(t)x+\int^t_0 \T^{cl}(t-s)f(s)ds,
	\end{equation}
	for any $t\ge 0,\;x\in X$ and $f\in L^p(\R^+,X)$. \\\\Before giving  another useful expression of $z$ in term of the intial semigroup $\T$, we need the following very useful result proved in \cite[prop.3.3]{Hadd}.
	\begin{lemma}\label{Hadd_lemma}
		Let $(\Ss(t))_{t\ge 0}$ be a strongly continuous semigroup on $X$ with generator $(\mathbb{G},D(\mathbb{G}))$. Let $\Upsilon\in\calL(D(\mathbb{G}),X)$ be a $p$-admissible observation operator for $\mathbb{G}$. Denote by $\Upsilon_\Lambda$ the Yosida extension of $\Upsilon$ with respect to $\mathbb{G}$. Then
		\begin{align*}
		&(\Ss\ast f)(t):=\int^t_0 \Ss(t-s)f(s)ds\in D(\Upsilon_\Lambda),\quad \text{a.e.}\;t\ge 0,\cr
		&\left\|\Upsilon_\Lambda (\Ss\ast f)\right\|_{L^p([0,\alpha],X)} \leq c(\alpha)\|f \|_{L^p([0,\alpha],X)},
		\end{align*}
		for $\alpha >0$, $f\in L_{loc}^p([0,\infty),X)$ and a constant $c(\alpha)$ independent of $f$ such that $c(\alpha)\rightarrow 0$ as $\alpha\rightarrow 0$.
	\end{lemma}
	\begin{proposition}\label{VCF-calA-f}
		Let assumptions of Theorem \ref{theorem:Hadd_Manzo_Rhandi} be satisfied. Let $\al>0$ and $f\in L^p(\R^+,X)$.  The non-homogenous Cauchy problem \eqref{calA-only} is well-posed and its mild solution satisfies for any initial condition $x\in X,$
		\begin{align}\label{3-calA-f-z}
		\begin{split}
		&z(t)\in D(C_\Lambda)\quad\text{a.e.}\; t>0,\cr & \|C_\Lambda z(\cdot)\|_{L^p([0,\al],X)}\le c(\al)  \left(\|x\|+\|f\|_{L^p([0,\al],X)}\right),\cr & z(t)=\T(t)x+\int^t_0 \T_{-1}(t-s)BC_\Lambda z(s)ds+\int^t_0 \T(t-s)f(s)ds,\qquad t\ge 0,
		\end{split}
		\end{align}
		where $c(\al)>0$ is a constant independent of $f$.
	\end{proposition}
	\begin{proof}
		By virtue of Theorem  \ref{theorem:Hadd_Manzo_Rhandi}, $\T^{cl}$ the semigroup generated by $\calA$ and let $z:[0,+\infty)\to X$ be the mild solution of the problem \eqref{calA-only} given by \eqref{formula-used}. According to Theorem \ref{staff_wei}, we know that $C_\Lambda$ is an admissible observation operator for $\calA$. We denote by $C_{\Lambda,\calA}$ the Yosida extension of $C_\Lambda$ with respect to $\calA$. Then $D(C_{\Lambda,\calA}) \subset D(C_\Lambda)$ and $C_{\Lambda,\calA}=C_\Lambda$ on $D(C_{\Lambda,\calA})$. In fact, let $x\in D(C_{\Lambda,\calA})$ and $s>0$ sufficiently large. Then by first taking Laplace transform on both sides of \eqref{WS-VCF} and second applying $sC_\Lambda,$ we obtain
		\begin{equation}\label{CR-formula}
		sC_\Lambda R(s,\calA)x=sCR(s,A)+H(s)C_\Lambda s R(s,\calA)x,
		\end{equation}
		where we have used \eqref{transfer-function-ABC}. Remark that
		\begin{equation*}
		\|H(s)C_\Lambda s R(s,\calA)x\|\le \|H(s)\|\left(\|C_\Lambda s R(s,\calA)x-C_{\Lambda,\calA}x\|+ \|C_{\Lambda,\calA}x\|\right).
		\end{equation*}
		Hence, by \eqref{limit-trans-func} and the fact that $x\in D(C_{\Lambda,\calA})$, we obtain
		\begin{equation*}
		\lim_{s\to+\infty}H(s)C_\Lambda s R(s,\calA)x=0.
		\end{equation*}
		Now from \eqref{CR-formula}, we deduce that $x\in D(C_\Lambda)$ and $C_{\Lambda,\calA}x=C_\Lambda x$. Let $x\in X,\;\al>0$ and $f\in L^p([0,\al],X)$. The fact that $C_\Lambda$ is $p$--admissible for $\calA,$ then by using \eqref{formula-used} and Lemma \ref{Hadd_lemma}, we obtain $z(t)\in D(C_{\Lambda,\calA})$ for  a.e. $t>0$.  This shows that $z(t)\in D(C_\Lambda)$ and $C_\Lambda z(t)=C_{\Lambda,\calA}z(t)$ for a.e. $t>0$. The estimation in \eqref{3-calA-f-z} follows immediately from \eqref{T-cl-estim} and Lemma \ref{Hadd_lemma}. Let us prove the last property in \eqref{3-calA-f-z}. By density there exists $(f_n)_n\subset \calC([0,\al],D(\calA))$ such that $f_n\to f$ in $L^p([0,\al],X)$ as $n\to \infty$. We put
		\begin{equation}\label{expr-z-n}
		z_n(t)=\T^{cl}(t)x+\int^t_0 \T^{cl}(t-s)f_n(s)ds,\qquad t\ge 0.
		\end{equation}
		Using H\"{o}lder inequality, it is clear that $\|z_n(t)-z(t)\|\to 0$ as $n\to\infty$. Now let us prove that $z_n$ satisfies the third assertion in \eqref{3-calA-f-z}. In fact, the estimation in \eqref{3-calA-f-z} implies that
		\begin{equation*}
		\|C_\Lambda z_n(\cdot)-C_\Lambda z(\cdot)\|_{L^p([0,\al],X)}\le c_\al \|f_n-f\|_{L^p([0,\al],X)}\underset{n\to\infty}{\longrightarrow}0.
		\end{equation*}
		On the other hand, using the expression of the semigroup $\T^{cl}$ given in \eqref{WS-VCF}, change of variable and Fubini theorem we obtain
		\begin{align}\label{z_n}
		z_n(t)&=\T^{cl}(t)x+\int^t_0\T(t-s)f_n(s)ds+\int^t_0 \T_{-1}(t-\t)B\int^{\t}_0 C_\Lambda \T^{cl}(\t-s)f_n(s)ds d\t \cr &=\T(t)x+ +\int^t_0\T(t-s)f_n(s)ds\cr & \hspace{1.5cm}+\int^t_0 \T_{-1}(t-\t)B\left(C_\Lambda \T^{cl}(\t)x+\int^{\t}_0 C_\Lambda \T^{cl}(\t-s)f_n(s) ds\right) d\t.
\end{align}
		For simplicity we assume that $0\in\rho(\calA)$. Then we have
		\begin{align*}
		C_\Lambda \int^{\t}_0 \T^{cl}(\t-s)f_n(s)ds&=C_\Lambda (-\calA)^{-1} (-\calA)\int^{\t}_0 \T^{cl}(\t-s)f_n(s)ds \cr &=C_\Lambda (-\calA)^{-1} \int^{\t}_0 \T^{cl}(\t-s)(-\calA)f_n(s)ds \cr &=\int^{\t}_0 C_\Lambda (-\calA)^{-1} \T^{cl}(\t-s)(-\calA)f_n(s)ds\cr &=\int^{\t}_0 C_\Lambda \T^{cl}(\t-s)f_n(s)ds.
		\end{align*}
		Now replacing this in \eqref{z_n}, and using \eqref{expr-z-n}, we obtain
		\begin{equation*}
		z_n(t)=\T(t)x+\int^t_0 \T_{-1}(t-s)BC_\Lambda z(s)ds+\int^t_0 \T(t-s)f_n(s)ds,\qquad t\ge 0.
		\end{equation*}
		Put
		\begin{equation*}
		\varphi(t)=\T(t)x+\int^t_0 \T_{-1}(t-s)BC_\Lambda z(s)ds+\int^t_0 \T(t-s)f(s)ds,\qquad t\ge 0.
		\end{equation*}
		Then for any $t\in [0,\al],$ we have
		\begin{equation*}
		\|z_n(t)-\varphi(t)\|\le \ga_\al\left( \|C_\Lambda z_n(\cdot)-C_\Lambda z(\cdot)\|_{L^p([0,\al],X)}+ \|f_n-f\|_{L^p([0,\al],X)}\right),
		\end{equation*}
		due to the admissibility of $B$ for $A$ and  H\"{o}lder inequality. This shows that $\|z_n(t)-\varphi(t)\|\to 0$ as $n\to \infty$, and hence $z=\varphi$.
	\end{proof}
	Now we can state the main result of this section.
	\begin{theorem}\label{theorem:generalisation_of_Hadd_Manzo_Rhandi}
		Let assumptions of Theorem \ref{theorem:Hadd_Manzo_Rhandi} be satisfied. In addition, let $P:Z\to X$ such that  $(A,B,\P)$  generates a regular linear system on $X,U,X$, where $\P=Pi$. The following assertions hold:
		\begin{itemize}
			\item [{\rm (i)}] The operator $P\in\calL(D(\calA),X)$ is a $p$-admissible observation operator for $\calA$, hence the operator $(\calA+P,D(\calA))$ generates a strongly continuous semigroup on $X$.
			\item [{\rm (ii)}] The boundary problem \eqref{PM} is well-posed and its mild solution $z:[0,+\infty)\to X$ satisfies:
			\begin{align*}
			& z(t)\in D(\P_\Lambda)\quad\text{a.e.}\; t\ge 0,\cr & z(t)=\T^{cl}(t)x+\int^t_0 \T^{cl}(t-s)\left(\P_\Lambda z(s)+f(s)\right)ds,
			\end{align*}
			for any $t\ge 0,$ initial condition $x\in X$ and $f\in L^p([0,\infty),X),$ where $\P_\Lambda$ is the Yosida extension of $P$ w.r.t $\calA$.
		\end{itemize}
	\end{theorem}
	\begin{proof}
		(i) We first remark from \eqref{relation-K-C-Lambda} that $Z\subset D(\P_{0,\Lambda})$ and $P=\P_{0,\Lambda}$ on $Z,$ where  $\P_{0,\Lambda}$ denotes the Yosida extension of $\P$ w.r.t. $A$. Let $x\in D(\calA)$ and $\al>0$. The facts that $(A,B,\P)$ is regular and using \eqref{T-cl-estim}, we obtain
		\begin{align}\label{hhhhh}
		&\int^t_0 \T_{-1}(t-s)BC_\Lambda \T^{cl}(s)x\in D(\P_{0,\Lambda})\quad\text{a.e.}\;t\ge 0,\;\text{and}\cr & \left\| \P_{0,\Lambda}\int^{\cdot}_0 \T_{-1}(t-s)BC_\Lambda \T^{cl}(s)x\right\|_{L^p([0,\al],X)} \le \beta(\al) \|x\|,
		\end{align}
		for some constant $\beta(\al)>0$ depending only on $\al$. \\On the other hand, appealing to \eqref{WS-VCF}, we obtain
		\begin{align*}
		P\T^{cl}(t)x&=\P_{0,\Lambda}\T^{cl}(t)x\cr &=\P_{0,\Lambda}\T(t)x+\P_{0,\Lambda}\int^{t}_0 \T_{-1}(t-s)BC_\Lambda \T^{cl}(s)x.
		\end{align*}
		Combining \eqref{hhhhh} and the $p$-admissibility of $\P$ for $A$ yield the $p$-admissibility of $P$ for $\calA$. According to Remark \ref{Hadd_thm} (i), the operator $(\calA+P,D(\calA))$ generates a strongly continuous semigroup on $X$.  The assertion (ii) follows from \cite[thm.5.1]{Hadd}
	\end{proof}
\section{Perturbation Theorems for maximal regularity}\label{sec:4}
\subsection{Maximal regularity}\label{sub:4.0}
Let $\G:D(\G)\subset X\to X$ be the generator of a strongly continuous semigroup $\Ss:=(\Ss(t))_{t\ge 0}$ on a Banach space $X$. Consider the following non--homogeneous abstract Cauchy problem:
	\begin{equation}
	\label{ACP}
	\begin{cases}
	\dot{z}(t) = \G z(t)+f(t) , & \mbox{ } 0<t \leq T \\
	z(0) = 0, & \mbox{}
	\
	\end{cases}
	\end{equation} where $f: [0,T]\rightarrow X$ a measurable function.
	\begin{definition}
		We say that the operator $\G$ (or the problem \eqref{ACP}) has the maximal $L^p$-regularity on the interval $[0,T]$, and we write $\G\in \mathscr{MR}_p(0,T;X)$, if for all $f\in L^p([0,T],X)$, there exists a unique $z \in W^{1,p}([0,T],X)\cap L^p([0,T],D(\G))$ which verifies (\ref{ACP}).
	\end{definition}
	By "maximal" we mean that the applications $f$, $\G z$ and $z$ have the same regularity. Due to  the closed graph theorem,  if  $\G\in \mathscr{MR}_p(0,T;X)$ then
	\begin{equation}
	\label{max_reg_estim}
	\Vert \dot{z}\Vert_{L^p([0,T],X)}+\Vert z\Vert_{L^p([0,T],X)}+\Vert \G z\Vert_{L^p([0,T],X)}\leq C\Vert f\Vert_{L^p([0,T],X)},
	\end{equation}
	for a constant $C>0$ independent of $f$.
	
	It is well-known that a necessary condition for the maximal $L^p$-regularity is that $\G$ generates an analytic semigroup. According to De Simon \cite{Simon} this condition is also sufficient if $X$ is a Hilbert space. On the other hand,  it is shown in \cite{Dore} that if  $\G\in \mathscr{MR}_p(0,T;X)$ for one $p\in [1,\infty]$ then $\G\in \mathscr{MR}_q(0,T;X)$ for all $q\in]1,\infty[$. Moreover if $\G\in \mathscr{MR}_p(0,T;X)$ for one $T>0$, then $\G\in \mathscr{MR}_p(0,T';X)$ for all $T'>0$. Hence we simply write $\G\in \mathscr{MR}(0,T;X)$.
\begin{remark}\label{remark:Maximal_Regularity_Operator}
		\begin{enumerate}
			\item[{\rm(i)}] Let  $\calC([0,T];D(\G))$ be the space of all continuous functions from $[0,T]$ to $D(\G)$, which  is a dense space of $L^p([0,T];X)$. It is well-known ((see \cite{LutzWeis} (2.a) or \cite{KunstmannWeis} 1.5)) that  $\G$ has maximal $L^p$-regularity on $[0,T]$ if and only if $(\Ss(t))_{t\geq 0}$ is analytic and the operator $\mathscr{R}$ defined by
			\begin{equation}
			\label{R_op}
			(\mathscr{R}f)(t):=\G\int_0^t \Ss(t-s)f(s)ds \qquad \qquad f\in \calC([0,T],D(\G)),
			\end{equation}
			extends to a bounded operator on $L^p([0,T];X)$. As we will see in our main results, this characterization is very useful if one works in general Banach spaces.
			\item [{\rm(ii)}] It is known (see \cite{Dore}) that if $\G\in \mathscr{MR}(0,T;X)$ then for every $\lambda\in \mathbb{C}$, $\G+\lambda\in \mathscr{MR}(0,T;X)$, hence without lost of generality, we will assume through this paper that our generators satisfy $\omega_0(\G)<0$.
		\end{enumerate}
	\end{remark}
	In order to recall another characterization of maximal regularity, we need some preliminary definitions.
\begin{definition}
		We say that a Banach space $X$ is a UMD-space if for some (hence all) $p\in (1,\infty)$ , $\mathscr{H}\in \mathcal{L}(L^p(\mathbb{R},X))$ where
$$
(\mathscr{H}f)(t) = \frac{1}{\pi} \lim_{\epsilon\to 0} \int_{\vert s\vert>\epsilon} \frac{f(t-s)}{s}ds,\quad t\in \mathbb{R},\quad  f\in \mathcal{S}(\mathbb{R},X),
		$$
		where $\mathcal{S}(\mathbb{R},X)$ is the Schwartz space.
	\end{definition}
	Classical  UMD-spaces are Hilbert spaces and $L^p$-spaces, where $p\in (1,\infty)$. It is to be noted that every  UMD-space is a reflexive space (see \cite{Amann}).
	\begin{definition}
		A set $\tau \subset \mathcal{L}(X,Y)$ is called $\mathcal{R}$-bounded if there is a constant $C>0$ such that for all $n\in \mathbb{N}$, $T_1,...,T_n \in \tau$, $x_1,...,x_n \in X$,
		\begin{equation*}
		\int_0^1 \Vert \sum_{j=1}^n r_j(s) T_j x_j \Vert_Y ds \leq C \int_0^1 \Vert \sum_{j=1}^n r_j(s) x_j \Vert_X ds,
		\end{equation*}
		where $(r_j)_{j\geq 1}$ is a sequence of independent $\{-1;1\}$-valued random variables on $[0,1]$(e.g. Rademacher variables).
	\end{definition}
\begin{remark}\label{Transfer_Function_R-boundedness}
		Here we give examples of $\mathcal{R}$-bounded sets.  We let $A$ be the generator of a bounded  analytic semigroup on a Banach space $X$, $B:U\to X_{-1}$ and $C:D(A)\subset X\to U$ are linear bounded operators, where $U$ is another (boundary) Banach space. We assume that $(A,B,C)$ generates a regular linear system on $X,U,U$ with transfer function
		\begin{equation*}
		H(\la):=C_\Lambda R(\la,A_{-1})B\in\calL(U),\qquad \la\in\rho(A),
		\end{equation*}
		where $C_\Lambda$ is the Yosida extension of $C$ with respect to $A$, see Section \ref{sec:2}. It has been observed in \cite[p.513]{HaaKun} that the set $\{H(is):s\neq 0\}$ is $\mathcal{R}$-bounded.
	\end{remark}
	The following remarkable result is due to Weis \cite{LutzWeis2}
	\begin{theorem}\label{theorem:Weis_Characterization}
		Let $\G$ be the generator of a bounded analytic semigroup in a UMD-space $X$. Then $\G$ has maximal $L^p$-regularity for some (hence all) $p\in (1,\infty)$ if and only if the set $\{sR(is,\G);s\neq 0 \}$ is $\mathcal{R}$-bounded.
	\end{theorem}
\subsection{Robustness of analyticity under Staffans-Weiss perturbation}\label{sub:4.3}
	In this part, we study the analyticity of the perturbed semigroup $(\T^{cl}(t))_{t\geq 0}$ defined in (\ref{theorem:Hadd_Manzo_Rhandi}). We then assume, as in Section \ref{sec:3}, that {\bf(H1)}, {\bf(H2)} and {\bf(H3)} are satisfied.
\begin{theorem}\label{theorem:staffans-weiss-analytic}
		Let assumptions {\bf(H1)}to {\bf(H3)} be satisfied. Then the operator $(\calA,D(\calA))$ defined by \eqref{cald-def} generates a strongly continuous semigroup which is analytic whenever the semigroup generated by $A$ is.
	\end{theorem}
	\begin{proof}
		According to Theorem \ref{theorem:Hadd_Manzo_Rhandi} (i) $\calA$ is a generator of a strongly continuous semigroup $\T^{cl}:=(\T^{cl}(t))_{t\ge 0}$ on $X$. Now assume that $A$ generates an analytic semigroup $\T$ on $X$. Then by \cite[Theorem 12.31]{ReRo}, there exist constants $\beta\in \mathbb{R}$ and $M_1>0$ such that $\mathbb{C}_\beta\subset \rho(A)$ and
		\begin{equation}\label{analytic-A-estim}
		\left\|R(\lambda,A)\right\| \leq \frac{M_1}{|\lambda-\beta|} ,\qquad \lambda\in \mathbb{C}_\beta.
		\end{equation}
		In \cite[Lemma 2.3]{LeMerdy}, it is proved that the admissibility of $C$ (and $B$ by the same method) for $A,$ imply that there exists $\ga \in ]\frac{\pi}{2},\pi[ $ such that
		\begin{align}\label{LeMerdy}
		\sup_{z\in\Si_\ga}\vert z\vert^\frac{1}{q} \Vert CR(z,A) \Vert < +\infty\quad\text{and}\quad \sup_{z\in\Si_\ga}\vert z\vert^\frac{1}{p} \Vert R(z,A_{-1})B \Vert < +\infty,
		\end{align}
		where $\frac{1}{p}+\frac{1}{q}=1$. \\Put		
		\begin{align}\label{Good-estim-Analy}
		M_2:=\sup_{z\in\mathbb{C}_0}\vert z\vert^\frac{1}{q} \Vert CR(z,A) \Vert < +\infty\quad\text{and}\quad M_3:=\sup_{z\in\mathbb{C}_0}\vert z\vert^\frac{1}{p} \Vert R(z,A_{-1})B \Vert < +\infty,
		\end{align}
		Let $\om_1:=\max\{\om_0(A),\om_0(\calA)\}$. Due to Theorem \ref{theorem:Hadd_Manzo_Rhandi} (ii), for any $\la\in\C_{\om_1},$ we have \begin{align}\label{Resolvent-calA}
		R(\la,\calA)&=(I-\D_\la K)^{-1}R(\la,A)\cr &= R(\la,A)+R(\la,A_{-1})B (I_U-C_\Lambda \D_\la)^{-1}CR(\la,A).
		\end{align}
		On the other hand, $(I_U-C_\Lambda \D_\la)^{-1}=I+\mathscr{H}^{cl}(\la),$ where $\mathscr{H}^{cl}$ is the transfer function of the (closed-loop) regular linear system generated by $(\calA,B,C_\Lambda)$. Hence there exists $\al>\om_0(\calA)$ such that
		\begin{equation}\label{transfer-estim}
		\nu:=\sup_{\C_\al}\left\|(I_U-C_\Lambda \D_\la)^{-1}\right\|<\infty.
		\end{equation}
		Now let $\om_2:=\max\{0,\alpha,\beta,\om_1\}$. Combining \eqref{analytic-A-estim}, \eqref{Good-estim-Analy}, \eqref{Resolvent-calA} and \eqref{transfer-estim}, we obtain
		\begin{equation*}
		\|R(\la,\calA)\|\le \frac{\tilde{M}}{|\la-\om_2|},
		\end{equation*}
		for all $\la\in\C_{\om_2}$, where $\tilde{M}=M_{1}+\nu M_{2}M_{3}$. Finally, by virtue of \cite[Theorem 12.31]{ReRo} the semigroup generated by $\calA$ is analytic.
	\end{proof}
\subsection{Staffans-Weiss perturbation for maximal regularity in UMD spaces}\label{sec:umd}
	We are still working under the setting of Section \ref{sec:3} and we will study the maximal regularity of the operator $\calA$ in the case of UMD spaces.
\begin{theorem}\label{theorem:staffans-weiss-R-boundedness}
		Let conditions {\bf(H1)} to {\bf(H3)} be satisfied with $X,U$ be UMD-spaces and $A$ generates a bounded analytic semigroup. Assume  that there exists $\om>\om_1$ such that the sets $\{s^\frac{1}{p} R(\om+is,A_{-1})B;s\neq 0\}$ and $\{s^\frac{1}{q} C R(\om+is,A);s\neq 0\}$ are $\mathcal{R}$-bounded. If $A\in \mathscr{MR}(0,T;X)$ then $\calA \in \mathscr{MR}(0,T;X)$.
	\end{theorem}
	\begin{proof}
		Assume that $A\in \mathscr{MR}(0,T;X)$ and let $\om>\om_1$ such that the sets $\{s^\frac{1}{p} R(\om+is,A_{-1})B;s\neq 0\}$ and $\{s^\frac{1}{q} C R(\om+is,A);s\neq 0\}$ are $\mathcal{R}$-bounded. Denote $\calA^\om:=-\om+\calA$ and $A^\om:=-\om+A$ with domains $D(\calA^\om)=D(\calA)$ and $D(A^\om)=D(A),$ respectively. These operators are generators of analytic semigroups on $X$. We first observe that $A^\om\in \mathscr{MR}(0,T;X)$. To prove our theorem it suffice to show that $\calA^\om\in \mathscr{MR}(0,T;X)$. Clearly, $\om_0(A^\om)=\om_0(A)-\om<0$ and $\om_0(\calA^\om)=\om_0(\calA)-\om<0$, so   that $i\R\backslash\{0\}\subset \rho(A^\om)\cap\rho(\calA^\om)$. It is not difficult to show that $(A^\om,B,C)$ is also a regular linear system on $X,U,U$ with the identity operator $I_U:U\to U$ as an admissible feedback. Now according to Theorem \ref{staff_wei}, the following operator
		\begin{equation*}
		A^{cl,\om}:=A^\om_{-1}+BC_\Lambda, \quad D(A^{cl,\om})=\left\{x\in D(C_\Lambda): (A^\om_{-1}+BC_\Lambda)x\in X\right\}.
		\end{equation*}
		As $A^\om_{-1}=A_{-1}-\om$, then $D(A^{cl,\om})=D(\calA),$ and $A^{cl,\om}=\calA^\om$ due to Theorem \ref{theorem:Hadd_Manzo_Rhandi} (i).
		As in \eqref{Resolvent-calA} we have
		\begin{align}\label{formu-is}
		\begin{split}
		sR(is,\calA^\om) &= sR(is,A^\om) + s^\frac{1}{p}R(is,A_{-1}^\om)B(I-H^\om(is))^{-1} s^\frac{1}{q} CR(is,A^\om)\cr & =sR(is,A^\om) + s^\frac{1}{p}R(\om+is,A_{-1})B(I-H^\om(is))^{-1} s^\frac{1}{q} CR(\om+is,A),
		\end{split}
		\end{align}
		where $H^\om(\la)=C_\Lambda R(\la,A_{-1}^\om)B,\;\la\in\rho(A)$, is the transfer function of the regular linear system generated by $(A^\om,B,C)$. Using the assumptions,  the equation \eqref{formu-is} and Theorem \ref{theorem:Weis_Characterization}, it suffice only to show that the set $\{(I-H^\om(is))^{-1}:s\neq 0\}$ is $\mathcal{R}$-bounded. In fact, by Theorem \ref{theorem:Hadd_Manzo_Rhandi} (i) and  the condition { \bf (H3)} the triple operator $(\calA^\om,B,C_\Lambda)$ generates a regular linear system and its associated transfer function is
		\begin{equation*}
		H^{cl,\om}(is) = (I-H^\om(is))^{-1}H^\om(is),\qquad s\neq 0,
		\end{equation*}
		which implies that
		\begin{equation*}
		(I-H^\om(is))^{-1}=I_U+H^{cl,\om}(is),\qquad s\neq 0.
		\end{equation*}
		According to Remark \ref{Transfer_Function_R-boundedness}, the set $\{H^{cl,\om}(is):s\neq 0\}$ is $\mathcal{R}$-bounded. Hence $\{(I-H^\om(is))^{-1}:s\neq 0\}$ is $\mathcal{R}$-bounded. This ends the proof.
	\end{proof}
\subsubsection{Example: PDE with a boundary conditions}
	Let $p,q\in(1,\infty)$ such that $\frac{1}{p}+\frac{1}{q}=1$ and $1<p<3$. In this section, we deal with the following problem:
	\begin{equation}
	\label{PDE}
	\begin{cases}
	\displaystyle{\frac{\partial w(t,s)}{\partial t}  = \frac{\partial^2 w(t,s)}{\partial s^2}} + f(t) & \mbox{ } t\in [0,T]\ s\in (0,1),  \\
	\displaystyle{\frac{\partial w(t,1)}{\partial s}} = w(t,1)-w(t,0), & \mbox{}  \\
	\displaystyle{\frac{\partial w(t,0)}{\partial s}} = 0, & \mbox{}  \\
	w(0,s) = 0. & \mbox{}
	\
	\end{cases}
	\end{equation}
Let $X=L^p(0,1)$ and $A_mg:=g''$ with domain $Z=\{g\in W^{2,p}(0,1),g'(0)=0 \}$. \\
	We define on $Z$ the operator:
	\begin{align*}
	G:Z&\to \mathbb{R} \\
	g&\mapsto g'(0).
	\end{align*}
	We also define the unbounded operator $K:Z\to \mathbb{R}$ by $Kg:=g(1)-g(0)$. One can see that under this setting, the problem (\ref{PDE}) can be transformed to the following :
	\begin{equation*}
	\begin{cases}
	\dot{x}(t) = A_m x(t) , & \mbox{ } 0\leq t \leq T,\quad x(0)=0 \\
	Gx(t) = Kx(t), & \mbox{  } 0\leq t \leq T.
	\
	\end{cases}
	\end{equation*}
	Define the operator:
	$$
	A=A_m,\qquad D(A)=Ker G=\{g\in W^{2,p}(0,1),g'(0)=0\quad and\quad g'(1)=0 \}.
	$$
	It is well-known that the operator $A$ has maximal$L^p$- regularity.\\\\
	We set:
	$$
	\mathbb{D}_\lambda := (G_{|ker(\lambda -A_m)})^{-1},\qquad B:=(\lambda - A_{-1})\mathbb{D}_\lambda, \qquad C :=Ki
	$$
	\\
	 \\
	Now the main result is the following
	\begin{theorem}
		The operator $A^{cl}$ defined by :
		$$
		A^{cl} g = g'',\qquad D(A^{cl}):=\{g\in W^{2,p}(0,1),g'(0)=0\quad and\quad g'(1)=g(1)-g(0) \}
		$$
		has maximal $L^r$-regularity for every $r>1$ and the problem (\ref{PDE}) has a unique solution $w\in W^{1,r}([0,T];X)\cap L^{r}([0,T];D(A^{cl}))$ such that:
		$$
		\Vert \dot{w} \Vert_{L^r([0,T];X)} + \Vert w \Vert_{L^r([0,T];X)} + \Vert w''\Vert_{L^r([0,T];X)} \leq C\Vert f \Vert_{L^r([0,T];X)},
		$$
		for some constant $C>0$ anf $f\in L^r([0,T];X)$.
	\end{theorem}
	\begin{proof}
		
		We will show that all conditions of Theorem \ref{theorem:staffans-weiss-R-boundedness} are satisfied. According to \cite[Theorem 2.4]{ABE}, to show that the triple $(A,B,C)$ generates a regular linear system, all we have to show is that there exist $\beta \in [0,1]$ and $\gamma \in (0,1]$ such that :
		
		\begin{enumerate}
			\item[(i)] $Range(\mathbb{D}_\lambda) \subset F_{1-\beta}^A$(where $F_{\alpha}^A$ is the Favard space of $A$ of order $\alpha$, see \cite{EngNag})
			\item[(ii)] $D[(\lambda - A)^\gamma] \subset D(C_\wedge)$
			\item[(iii)] $\beta + \gamma <1$.
		\end{enumerate}
		Let us prove the above three propositions:
		
		\item[(i)] By simple calculation, one can see that there exists $\lambda_0  > 0$ such that
		$$
		\sup_{\lambda>\lambda_0}\Vert \lambda^\frac{p+1}{2p} \mathbb{D}_\lambda \Vert <+\infty.
		$$
		This fact implies (see \cite[Lemma A.1]{ABE}) that $Range(\mathbb{D}_\lambda) \subset F_{\frac{p+1}{2p}}^A$ for some $\lambda\in \rho(A)$. We take $\beta = 1-\frac{p+1}{2p} = \frac{p-1}{2p}$.
		
		\item[(ii)] According to \cite[Lemma A.2]{ABE}, we have to show that for $\alpha \in (0,1)$ and every $\rho\geq \rho_0>0$ we have
		$$
		\vert Cg\vert \leq M(\rho^\alpha \Vert g\Vert_p + \rho^{\alpha-1}\Vert g''\Vert_p)\quad g\in D(A),
		$$
		for some constant $M>0$.\\ For $g\in D(A)$ we know that
		$$
		g(1)= g(0) + \int_0^1 g'(s)ds
		$$
		then we can easily show that :
		$$
		\vert g(1)-g(0)\vert \leq \Vert g'\Vert_p.
		$$
		By \cite[Example III.2.2]{EngNag}, we have for $\epsilon>0$
		$$
		\Vert g'\Vert_p \leq \frac{9}{\epsilon}\Vert g\Vert_p + \epsilon \Vert g''\Vert_p.
		$$
		By taking $\rho = \epsilon^{-3}$, $\alpha = \frac{1}{3}$ and $\gamma \in (\frac{1}{3},\frac{1}{p})$, we show the assertion.
		\item[(iii)] Clearly we have $\beta + \gamma <1$

		Hence, \cite[Theorem 2.4]{ABE} implies that the operators $B$ and $C$ are $r$-admissible for $\frac{2p}{p+1}<r<\frac{1}{\gamma}$ and the triple $(A,B,C)$ generates a regular system and the operator identity is an admissible feedback. For instance we can take $r=2$.
		
		To show that the sets $\{s^\frac{1}{2} CR(is,A);s\neq 0\}$ and $\{s^\frac{1}{2} R(is,A_{-1})B;s\neq 0\}$ are $\mathcal{R}$-bounded, it is sufficient to show that $C$ and $B$ are $l$-admissible ($l$-admissibility is more general than admissibility, see \cite{HaaKun} for definitions), namely, it is sufficient to show that the sets $\{\sqrt{s}CR(s,A);s>0\}$ and $\{\sqrt{s}R(s,A_{-1})B;s>0\}$ are $\mathcal{R}$-bounded (see \cite{HaaKun}, page 514), which is equivalent to show that the sets $\{\sqrt{s}CR(s+1,A);s>0\}$ and $\{\sqrt{s}R(s+1,A_{-1})B;s>0\}$ are $\mathcal{R}$-bounded
		
		In order to prove that $\{\sqrt{s}CR(s+1,A);s>0\}$ is $\mathcal{R}$-bounded, we follow the example in \cite{HaaKun} (page 528) and the technique used there to show the $\mathcal{R}$-boundedness of the above set, namely it suffices to find a space $\tilde{Z}$ such that $D(A)\subset \tilde{Z}\subset X$ and $C$ is bounded in $\Vert \cdot \Vert_{\tilde{Z}\to\mathbb{R}}$ and the set $\{\sqrt{s}R(s+1,A);s>0\}$ is $\mathcal{R}$-bounded in $\mathcal{L}(\tilde{Z},X)$, this holds, by the same example, for $\tilde{Z}=W^{1,p}(0,1)$ since $C$ is bounded in $\mathcal{L}(\tilde{Z},\mathbb{R})$.
		
		Now to show the $\mathcal{R}$-boundedness of $\{\sqrt{s}R(s+1,A_{-1})B;s>0\}$ we follow the same example, we show that the set $\{\sqrt{s}B^*R(s+1,A^*);s>0\}$ is $\mathcal{R}$-bounded. Let first determine $B^*$, we proceed again as in \cite{HaaKun} and multiplying the first equation in (\ref{PDE}) with a fixed $v\in C^\infty([0,1])$ and integrating by parts we obtain:
		$$
		<w'(t,\cdot),v>_{(0,1)} = <w(t,\cdot),v''>_{(0,1)} + Gw(t,\cdot)\delta_1 v,
		$$
		this means that $B^* = \delta_1$. By the same argument used to show that $\{\sqrt{s}CR(s+1,A);s>0\}$ is $\mathcal{R}$-bounded, we show that $\{\sqrt{s}B^*R(s+1,A^*);s>0\}$ is $\mathcal{R}$-bounded.
		
		All hypothesis of the theorem are satisfied, then the operator $A^{cl}$ defined by :
		$$
		A^{cl} g = g'',\qquad D(A^{cl}):=\{g\in W^{2,p}(0,1),g'(0)=0\quad and\quad g'(1)=g(1)-g(0) \}
		$$
		has maximal $L^2$-regularity, thus maximal $L^r$-regularity for every $r>1$ and then problem (\ref{PDE}) has unique solution $u\in W^{1,r}([0,T];X)\cap L^{r}([0,T];D(A^{cl}))$ such that:
		\begin{align*}
		\left[ \int_0^T\big (\int_0^1\Vert \frac{\partial w(t,s)}{\partial t} \Vert^p ds\big )^{r/p}dt\right]^{1/r} +&\left[ \int_0^T\big (\int_0^1\Vert w(t,s) \Vert^p ds\big )^{r/p}dt\right]^{1/r}\cr & + \left[ \int_0^T\big (\int_0^1\Vert\frac {\partial^2 w(t,s)}{\partial s^2} \Vert^p ds\big )^{r/p}dt\right]^{1/r}\cr & \leq C\left[ \int_0^T\big (\int_0^1\Vert f(t,s) \Vert^p ds\big )^{r/p}dt\right]^{1/r}
		\end{align*}
		for some constant $C>0$.
\end{proof}
	\begin{remark} Besides Hilbert spaces there is another interesting situation, where bounded sets of operators are automatically $\mathcal{R}$-bounded. For example if $F$ is a Banach space of type 2 and $E$ is a Banach space of cotype 2, then every bounded set $\tau\subset\calL(E,F)$ is $\mathcal{R}$-bounded (see. \cite{LT}). Thanks to (\ref{LeMerdy}), forl all $\om >\om_0(A)$ the sets $\{s^\frac{1}{2} CR(\om+is,A);s\neq 0\}$ and $\{s^\frac{1}{2} R(\om+is,,A_{-1})B;s\neq 0\}$ are bounded. So, to prove their $\mathcal{R}$-boundedness for the previous example, we distinguish two cases (case $p=2$ is trivial):
	\begin{itemize}
			\item[(i)] If $2< p<3$, then $X=L^p(0,1)$ is of cotype 2, hence the set $\{s^\frac{1}{2}C R(is,A);s\neq 0\}$ is $\mathcal{R}$-bounded since $U=\R$ is of type 2 and it suffices
			 only to show $\mathcal{R}$-boundedness of $\{s^\frac{1}{2} R(\om+is,A_{-1})B;s\neq 0\}$
			\item[(ii)] If $1< p<2$, then $X=L^p(0,1)$ is of type 2, hence the set $\{s^\frac{1}{2} R(\om+is,,A_{-1})B;s\neq 0\}$ is $\mathcal{R}$-bounded since $U=\R$ is of cotype 2 and
			it suffices that only to show that $\{s^\frac{1}{2} CR(is,A);s\neq 0\}$ is $\mathcal{R}$-bounded.
		\end{itemize}
\end{remark}
\subsection{Staffans-Weiss perturbation for maximal regularity in non-reflexive spaces}\label{sec:nonreflexive}
	In this section, we assume that $X$ is a non-reflexive Banach space. The following result shows the maximal regularity of the perturbed boundary value problem \eqref{PM} in the case of $P\equiv 0$.
	
	\begin{theorem}\label{theorem:Staffans_Weiss_Non_Reflexive_Case}
		Let assumptions {\bf(H1)} to {\bf(H3)} be satisfied. Assume additionally that there exists $\lambda_0 \in \mathbb{R}$ such that
		\begin{equation}\label{kappa-0}
		\kappa_0:=\sup_{\lambda>\lambda_0}\Vert \lambda \mathbb{D}_\lambda \Vert < +\infty.
		\end{equation} If $A\in \mathscr{MR}(0,T;X)$ then $\calA\in \mathscr{MR}(0,T;X)$.
	\end{theorem}
	
	\begin{proof} As $A\in \mathscr{MR}(0,T;X)$, there exists $\mathscr{R}\in \calL(L^p([0,T],X))$ such that
		\begin{equation*}
		(\mathscr{R}f)(t)=A\int^t_0 \T(t-s)f(s)ds,
		\end{equation*}
		for all $f\in L^p([0,T],X)$ and a.e $t\ge 0$. By Theorem \ref{theorem:staffans-weiss-analytic}, $\calA$ generates an analytic semigroup $\T^{cl}$ on $X$. We then can define the following operator
		$$
		(\mathscr{R}^{cl}f)(t):=\calA\int_0^t \T^{cl}(t-s)f(s)ds \qquad \qquad f\in C([0,T];D(\calA)).
		$$
		Our objective is to show that  the operator $\mathscr{R}^{cl}$ admits a bounded extension on $L^p([0,T];X)$. Let us consider the Yosida approximation operators of $\calA$ by
		\begin{equation*}
		\calA_n := n\calA R(n,\calA)= n^2 R(n,\calA) -nI,
		\end{equation*}
		for any $n\in\N$ such that $n>\om_0(\calA)$.\\ From  \eqref{Resolvent-calA} and for any sufficiently large integer $n$,
		one can write
		\begin{equation}\label{calAn}
		\calA_n = nAR(n,A) +n^2 \mathbb{D}_n(I-C_\wedge \mathbb{D}_n)^{-1}CR(n,A)\quad .
		\end{equation}
		Consider the following operators
		$$
		(\mathscr{R}^{cl}_nf)(t):=\calA_n\int_0^t \T^{cl}(t-s)f(s)ds,
		$$
		for $f\in C([0,T];D(\calA))$ and $n\in\N$ such that $n>\om_0(\calA)$. \\We have (see \cite{EngNag}),
		\begin{equation*}
		(\mathscr{R}^{cl}f)(t)=\lim_{n\to\infty} (\mathscr{R}^{cl}_nf)(t),
		\end{equation*}
		for every $t\in [0,T]$. \\ Using Proposition \ref{VCF-calA-f}, we obtain
		\begin{equation*}
		(\mathscr{R}^{cl}_nf)(t)=\calA_n \left(\int^t_0 \T(t-s)f(s)ds+\int^t_0 \T_{-1}(t-s)C_\Lambda z(s)ds\right),
		\end{equation*}
		where
		\begin{equation*}
		z(t)=\int^t_0 \T^{cl}(t-s)f(s)ds,\qquad t\ge 0.
		\end{equation*}
		Using  \eqref{calAn}, for large $n$,
		\begin{align*}
		(\mathscr{R}^{cl}_n f)(t) &=  nAR(n,A)\int_0^t \T(t-s)f(s)ds\\
		&+n^2 \mathbb{D}_n(I-C_\Lambda \mathbb{D}_n)^{-1}CR(n,A)\int_0^t \T(t-s)f(s)ds  \\
		&+nAR(n,A)\int_0^t \T_{-1}(t-s)BC_\Lambda z(s)ds\\
		&+n^2 \mathbb{D}_n(I-C_\Lambda \mathbb{D}_n)^{-1}CR(n,A)\int_0^t \T_{-1}(t-s)BC_\Lambda z(s)ds\\
		&:=I_n^1(t) + I_n^2(t) + I_n^3(t) + I_n^4(t).
		\end{align*}
		Let us estimate separatly the four terms on the left of the above equality. Maximal regularity of $A$ yields
		\begin{align*}
		\int_0^T \Vert I_n^1(t) \Vert^p dt &=\left\|\mathscr{R}( nR(n,A)f ) \right\|_{L^p([0,T],X)}^p\\
		&\leq c_T \left\| f \right\|_{(L^p[0,T],X)}^p,
		\end{align*}
		for a constant $c_T>0$ independent of $f$. \\On the other hand, combining  \eqref{transfer-estim}, \eqref{kappa-0} and  Lemma \ref{Hadd_lemma} we obtain
		\begin{align*}
		\int_0^T \|I_n^2(t) \|^p dt &=\int_0^T \left\| n \mathbb{D}_n(I-C_\wedge \mathbb{D}_n)^{-1}CnR(n,A)\int_0^t \T(t-s)f(s)ds \right\|^p dt\\ & \le (\kappa \nu)^p \int_0^T \left\|  C\int_0^t \T(t-s)nR(n,A)f(s)ds \right\|^p dt\\
		&\leq (\kappa \nu \ga_T)^p \|f\|_{L^p}^p:=c_1(T,p)  \|f\|_{L^p}^p,
		\end{align*}
		where $c_1(T,p)$ is a constant depending on $T$ and $p$.
		\\Now we estimate $I^3_n(t)$
		\begin{align*}
		\int_0^T \Vert I_n^3(t) \Vert^p dt &=\int_0^T \left\Vert A\int_0^t \T(t-s)n\D_n C_\Lambda z(s)ds \right\Vert^p dt\\
		&=\left\|\mathscr{R}(n\D_n C_\Lambda z(\cdot)) \right\|_{L^p([0,T],X)}^p\\ & \le c_T \kappa_0^p \left\| C_\Lambda z(\cdot) \right\|_{L^p([0,T],U)}^p \\
		&\leq c_2(T,p)   \Vert f \Vert_{L^p}^p,
		\end{align*}
		due to \eqref{kappa-0} and Proposition \ref{VCF-calA-f}, where $c_2(T,p)>0$ is a constant depending only on $T$ and $p$.
		\\ Similarly, we have
		\begin{align*}
		\int_0^T \Vert I_n^4(t) \Vert^p dt &=\int_0^T \left\Vert n \mathbb{D}_n(I-C_\Lambda \mathbb{D}_n)^{-1}C\int_0^t \T(t-s)nR(n,A)BC_\Lambda z(s)ds \right\Vert^p dt\\ &\le (\nu \kappa_0)^p\int_0^T \left\Vert  C\int_0^t \T(t-s)n \mathbb{D}_nC_\Lambda z(s)ds \right\Vert^p dt\\
		&\leq c_3(T,p) \Vert f \Vert_{L^p}^p,
		\end{align*}
		by \eqref{transfer-estim}, \eqref{kappa-0}, Lemma \ref{Hadd_lemma} and Proposition \ref{VCF-calA-f}, where $c_3(T,p)>0$ is a constant depending only on $T$ and $p$. Finally one can conclude that
		\begin{equation*}
		\int_0^T \Vert (\mathscr{R}_n^{cl} f)(t) \Vert^p dt \leq \tilde{C_p} \Vert f \Vert_{L^p}^p,
		\end{equation*}
		for some constant $\tilde{C(T,p)}>0$ depending only on $T$ and $p$.\\
		The fact that $\Vert (\mathscr{R}_n^{cl} f)(t) \Vert^p \rightarrow \Vert (\mathscr{R}^{cl} f)(t) \Vert^p$ for all $t\in [0,T]$, Fatou's Lemma implies
		\begin{align*}
		\int_0^T \Vert (\mathscr{R}^{cl} f)(t) \Vert^p dt &\leq \underset{n\to\infty}{\liminf} \int_0^T \Vert (\mathscr{R}_n^{cl} f)(t) \Vert^p dt\\
		&\leq \tilde{C_p} \Vert f \Vert_{L^p}^p.
		\end{align*}
		Thus $\mathscr{R}^{cl}$ can be extended to a bounded operator on $L^p([0,T];X)$.
	\end{proof}
\begin{remark} In the case of Desch-Schappacher perturbation (bounded perturbation at the boundary) we have obtained in \cite[Theorem ]{AmBoDrHa} the maximal regularity without appealing to the condition (\ref{kappa-0}). However for Staffans-Weiss perturbation (unbounded perturbation at the boundary), our proof of maximal regularity is essentially based on this condition. We mention that (\ref{kappa-0}) cannot take place if we work in a reflexive Banach space. In fact, this estimation can be expressed as   $\displaystyle{\sup_{\lambda>\lambda_0}}\Vert \lambda R(\lambda,A_{-1})B \Vert < +\infty$ which is equivalent to $Range(B) \subset F_1^{A_{-1}}$ (see \cite[Remark 10]{MarBouFadHam}, \cite{BMH}). It is important to recall that the control operator $B$ is strictly unbounded, i.e. $Range(B)\cap X = \{0\}$, since it comes from the boundary (see \cite{Sala}). So, if $X$ is reflexive then $F_1^{A_{-1}}=X$ and consequently $Range(B)\subset X$ which is contradictory since $B$ is not identically nul. Thus all these facts force us to work with a non-reflexive Banach space
\end{remark}
Now we state the previous theorem in the case of $X$ being a  $GT$-space (e.g. if $X=L^1$ or $X=C(K)$), which is a non-reflexive space. Since we are not concerned in this paper with $GT$-spaces and $H^\infty$-calculus, we omit definitions: all details can be found in \cite{KalWei}.
\begin{corollary}
Let assumptions of Theorem \ref{theorem:Staffans_Weiss_Non_Reflexive_Case} be satisfied. Let $X$ or $X^*$ be a $GT$-space. If $A$ is an $H^\infty$-sectorial operator on $X$ with $\omega_H(A)<\frac{\pi}{2}$ then $\calA \in \mathscr{MR}(0,T;X)$.
\end{corollary}
\begin{proof}
		According the Theorem 7.5 in \cite{KalWei}, if $A$ is an $H^\infty$-sectorial operator on $X$ with $\omega_H(A)<\frac{\pi}{2}$ then $A$ has maximal $L^p$-regularity for all $1<p<\infty$, which implies by Theorem \ref{theorem:Staffans_Weiss_Non_Reflexive_Case} that $A^{cl} \in \mathscr{MR}(0,T;X)$.
	\end{proof}
We end this section with the following result given the maximal $L^p$-regularity for the evolution equation \eqref{PM} (or equivalently \eqref{calA-max}).
\begin{corollary}\label{SF-big-equa-result}
	Let conditions {\bf(H1)} to { \bf (H3)} be satisfied on Banach spaces $X,U$ and $A\in \mathscr{MR}(0,T;X)$. Moreover, we assume that $(A,B,\P)$ generates a regular linear system on $X,U,X$ with $\P\in\calL(D(A),X)$ is the restriction of $P:Z\to X$ on $D(A)$. Then the operator $\calA+P\in \mathscr{MR}(0,T;X)$ (or equivalently the evolution equation \eqref{PM} has maximal $L^p$-regularity) if one of the following conditions hold:
	\begin{itemize}
		\item [{\rm (i)}]$X$ is a non reflexive Banach space and there exists $\la_0\in\R$ such that
		\begin{equation*}
		\sup_{\lambda>\lambda_0}\Vert \lambda \mathbb{D}_\lambda \Vert < +\infty.
		\end{equation*}
		\item [{\rm (ii)}] $X$ and $U$ are UMD spaces and the sets $\{s^\frac{1}{p} R(\om+is,A_{-1})B;s\neq 0\}$ and $\{s^\frac{1}{q} C R(\om +is,A);s\neq 0\}$ are $\mathcal{R}$-bounded for $\om >\om_1$.
	\end{itemize}
\end{corollary}
\begin{proof}
	First of all the operator $P$ is a $p$-admissible observation operator for $\calA,$ due to Theorem \ref{theorem:generalisation_of_Hadd_Manzo_Rhandi}. Thanks to Theorem \ref{theorem:Staffans_Weiss_Non_Reflexive_Case} (resp. Theorem \ref{theorem:staffans-weiss-R-boundedness} ), the assertion (i) (resp. (ii) ) implies that $\calA$ has maximal regularity. Finally, appealing to \cite[Theorem 3]{AmBoDrHa} we obtain $\calA+P\in \mathscr{MR}(0,T;X)$ which completes the proof.
\end{proof}
\section{Applications to non-autonomous problem}\label{sec:5}
	The object of this section is to apply our obtained results to prove maximal regularity of some classes of non-autonomous problem. First, we present the definition of maximal $L^p$-regularity in the non-autonomous case.
	
	Let $X$ be a Banach space and $\{A(t), t\in [0,T]\}$ a family of closed linear operators on $X$. For a measurable function $f:[0,T]\to X$, we consider the non-autonomous Cauchy problem:
	\begin{equation}
	\label{NCP}
	\tag{NCP}
	\begin{cases}
	\dot{x}(t) = A(t)x(t) + f(t) , & \mbox{ } t\in [0,T] \\
	x(0)=0. & \mbox{}
	\
	\end{cases}
	\end{equation}
	\begin{definition}
		For $p\in (1,\infty)$, we say that $\{A(t), t\in [0,T]\}$ or \eqref{NCP} has $L^p$-\textit{maximal regularity} if for all $f\in L^p([0,T],X)$, there exists a unique solution of \eqref{NCP} $x\in W^{1,p}([0,T],X)$ such that $t\mapsto A(t)x(t) \in L^p([0,T],X)$.
	\end{definition}
	 Maximal $L^p$-regularity in the non-autonomous case is quite different and less understood. However, several results have been established. We review some of them separating between the case where $D(A(t))$ does not depend on $t$ and the case where $D(A(t))$ is varying with respect to $t$. In the first case, that of $D(A(t))=D(A(0))$, Pr\"uss and Schnaubelt \cite{PruSchn} and Amann \cite{AmannMax} proved that \eqref{NCP} has $L^p$-maximal regularity under the conditions that $t\mapsto A(t)$ is continuous and $A(t)$ has $L^p$-maximal regularity for every $t\in [0,T]$. Arendt et al. \cite{ArRaFoPo} generalized this result to relative continuous function $t\mapsto A(t)$. When the operators $A(t)$ have time-dependent domains and $X$ is Hilbert space, Hieber and Monniaux \cite{HieMon} showed that \eqref{NCP} has $L^p$-maximal regularity whenever the family $\{A(t), t\in [0,T]\}$ satisfies the so-called Acquistapace Terreni conditions and every $A(t)$ has $L^p$-maximal regularity. This result was extended by Portal and \v{S}trkalj \cite{PortStrk} to a UMD spaces by using the concept of $\calR$-boundedness.\\
	 \\
	 The following theorem, due to Pr\"uss and Schnaubelt \cite{PruSchn}, provides a characterization of maximal regularity by means of continuity.
\begin{theorem}\label{theorem:Pruss-Schnaubelt}
	 	Let $A(t)$, $t\in [0,T]$, generates an analytic semigroup of negative type on $X$, $D(A(t)) = D(A(0))=:X_1$, and $A(\cdot) \in C([0,T],\calL(X_1,X))$, where $X_1$ is endowed with the graph norm of $A(0)$. Assume additionally that for every $t\in [0,T]$,  $A(t)$ has maximal $L^p$-regularity. Then $\{A(t), t\in [0,T]\}$ has maximal $L^p$-regularity on $X$.
	 \end{theorem}
In the following, let $X$, $Z$ and $U$ be Banach spaces such that $Z\subset X$ with dense and continuous embedding. We consider the non-autonomous problem
 \begin{align*}
  \begin{cases} \dot{z}(t)=A_m(t) z(t)+f(t),& t\in [0,T],\cr z(0)=0,\cr Gz(t)=Kz(t),& t\in [0,T],\end{cases}\end{align*}
 where $A_m(t):Z\subset X\to X$, $\in [0,T]$, are linear closed operators in a Banach space $X$ with domain $D(A_m(t)):=Z$, $G,K:Z\to U$ are linear boundary operators and $f\in L^p(\R^+,X)$ with $p\ge 1$ is a real number. Define the operator
 \begin{equation*}
 	\calA(t):= A_m(t) \text{ with } D(\calA(t)):=\{x\in Z; Gx = Kx \}=:\calD.
 \end{equation*}
  We define the non-autonomous \textit{Dirichlet map}
 $$
 \D_\la(t) := (G_{|Ker(\la-A_m(t))})^{-1}\in \calL(U,X), \qquad \la\in\rho(A(t)),
 $$
 and the $\lambda$-independing operators
 \begin{equation}\label{B(t) operator}
 B(t):= (\lambda - A_{-1}(t))\mathbb{D}_\la(t) \in \mathcal{L}(U,X_{-1}),
 \end{equation}
 and
 \begin{equation*}
 	C:= Ki,
 \end{equation*}
 where $i:D\to Z$ is the canonical injection form $D$ to $Z$.

The following proposition asserts that continuity is invariant to cross boundary perturbation.
\begin{proposition}\label{theorem:continuity perturbation}
	Let assumptions ($\bf {H1}$) to ($\bf {H3}$) be satisfied for every $t\in [0, T]$ by the operators $A(t), B(t)$ and $C$. Assume addionally ($\bf H4$): There is $\mu_0 >\om_0(A(0))$ such that $Ker(\mu_0 - A_m(t))$ does not depend on $t$. If $A(\cdot) \in C([0,T],\calL(D,X))$ then $\calA(\cdot) \in C([0,T],\calL(\calD,X))$.
\end{proposition}
\begin{remark}
	\begin{itemize}
		\item[(i)] Assumption \textbf{(H4)} implies that $\D_{\mu_0}(t)$ does not depend on $t$ and we set $\D_{\mu_0} = \D_{\mu_0}(t)$ for all $t\in [0,T]$. Hence, from the discussion above
		\item[(ii)] As for autonomous case, the definition of $B(t)$ does not depend on $\lambda$, so we can consider $B(t) = (\mu_0-A_{-1}(t))\D_{\mu_0}$ for all $t\in [0, T]$. Thanks to Theorem \ref{theorem:Hadd_Manzo_Rhandi}, we have
		\begin{equation*}
		\calA(t) = A_{-1}(t)+B(t)C_\wedge\qquad\hbox{for all}\qquad t\in [0,T].
		\end{equation*}
		By injecting $B(t)$ into above, we obtain
		\begin{equation*}
		\calA(t)-\mu_0 = A_{-1}(t)-\mu_0+(\mu_0-A_{-1}(t))\D_{\mu_0}C_\wedge = (A_{-1}(t)-\mu_0)(I-\D_{\mu_0}C_\wedge).
		\end{equation*}
		A simple calculation leads to $x\in D(\calA(t))\Leftrightarrow (I-\D_{\mu_0}M)x\in D(A(t))$. Hence
		\begin{equation}\label{expression of calA}
		\calA(t)-\mu_0  = (A(t)-\mu_0)(I-\D_{\mu_0}C_\wedge).
		\end{equation}
	\end{itemize}
\end{remark}
\begin{proof} (of Proposition \ref{theorem:continuity perturbation})
	As consequences of the continuity of $A(\cdot)$ on $[0,T]$ (see \cite{PruSchn} page 409): the type of of the semigroup generated by $A(t)$ does not depend on $t$ and the graph norms of the operators $A(t)$ are uniformly equivalent (and so we equip $D$ with the norm $\|\cdot\|_D = \|(\mu_0- A(t))\cdot\|$). Thus, the completions of $X$ with respect to the norms $\|\cdot\|_{-1}^{A(t)}:=\|R(\la,A(t))\cdot \|$ with $t\in [0,T]$ and $\la\in \rho(A(t))$ are constant and denoted simply $X_{-1}$. Denote by $A_{-1}(t)$ the extension of $A(t)$ to $X_{-1}$. Assumptions $(\bf {H1}$), ($\bf {H2}$) and $(\bf {H3})$ and  Theorem \ref{theorem:Hadd_Manzo_Rhandi} imply that for all $t\in [0,T]$ the operator $\calA(t)$ generates a strongly continuous semigroup on $X$. Then we equip $\calD$ with the norm $\| \cdot \|_\calD := \|(\mu_0 - \calA(0))\cdot\|$. Put $\calP_{\mu_0}:= I-\D_{\mu_0}C_\wedge$ and define on $D$ the norm $\| \cdot \|_D := \|(\mu_0 - A(0))\cdot\|$. By virtue of \eqref{expression of calA} for all $x\in \calD$ we have
	\begin{align*}
		\|x\|_\calD &= \|(\mu_0 - \calA(0))x\|\cr
		&= \|(\mu_0 - A(0))\calP_{\mu_0}x\|\cr
		&= \|\calP_{\mu_0}x\|_D.\cr
	\end{align*}
	On the other hand, for every $t,s\in [0,T]$ we have
	\begin{align*}
	\|\calA(t) - \calA(s)\|_{\calL(\calD,X)} &= \sup_{\|x\|_\calD = 1} \|\calA(t)x - \calA(s)x\|\cr
	&=\sup_{\|x\|_\calD = 1} \|(A(t) - A(s))\calP_{\mu_0}x\|\cr
	&=\sup_{\|\calP_{\mu_0}x\|_D = 1} \|(A(t) - A(s))\calP_{\mu_0}x\|\cr
	&\leq \sup_{\|z\|_D = 1} \|(A(t) - A(s))z\|\cr
	& = \|A(t) - A(s)\|_{\calL(D,X)}.\cr
   \end{align*}
Thus, the continuity of $\calA(\cdot)$ can be deduced from the one of $A(\cdot)$.
\end{proof}
Similarly, the continuity of $\calA(\cdot)$ on $[0,T]$ (see \cite{PruSchn} page 409) yields the graph norms of the operators $\calA(t)$ are uniformly equivalent. In particular the type of of the semigroup generated by $\calA(t)$ does not depend on $t$.\\
\\
Now, we present our perturbation result for the non-autonomous maximal regularity.
\begin{theorem}\label{theorem:nonautonomous staffans weiss}
	Let assumptions of Proposition \ref{theorem:continuity perturbation} be satisfied. If for every $t\in [0,T]$, $A(t)$ has maximal $L^p$-regularity and $A(\cdot) \in C([0,T],\calL(D,X))$. Then the family $\{\calA(t), t\in [0,T] \}$ has maximal $L^p$-regularity if one of the following conditions holds:
\begin{itemize}
\item [{\rm (i)}]$X$ is a non reflexive Banach space and for every $t\in [0,T]$, there exists $\la_0\in\R$ such that
\begin{equation*}
\sup_{\lambda>{\lambda_0}}\Vert \lambda \mathbb{D}_\lambda(t) \Vert < +\infty.
\end{equation*}
\item [{\rm (ii)}] $X$ and $U$ are UMD spaces, the  sets \begin{equation*}\{s^\frac{1}{p} R(\om+is,A_{-1}(t))B(t);s\neq 0\}\;\text{and}\; \{s^\frac{1}{q} C R(\om+is,A(t));s\neq 0\},\; t\in [0,T],\end{equation*} are $\mathcal{R}$-bounded for some $\om >\max\{\om_0(A(0)),\om_0(\calA(0))\}$.
\end{itemize}
\end{theorem}
\begin{proof}
	Thanks to Proposition \ref{theorem:continuity perturbation}, we have $\calA(\cdot) \in C([0,T],\calL(\calD,X))$. If condition (i) (resp. (ii)) is satisfied then by virtue of Theorem \ref{theorem:Staffans_Weiss_Non_Reflexive_Case} (resp. Theorem \ref{theorem:staffans-weiss-R-boundedness}) for the operator $\calA(t)$ has maximal $L^p$-regularity every $t\in [0,T]$. Applying Theorem \ref{theorem:Pruss-Schnaubelt}, we obtain the desired result.
\end{proof}
\begin{remark}
	Let $A_m, A, \calA, B$ and $C$ be as in Section \ref{sec:4} and let assumptions of Theorem \ref{theorem:staffans-weiss-R-boundedness} be satisfied. We define the operators $A_m(t) = a(t)A_m$ , $A(t) = a(t)A$ and $\calA(t) = a(t) \calA$ where $a\in C([0,T])$ with $a(t)>\alpha>0$ for all $ t\in [0,T]$. Then clearly, all assumptions of Theorem \ref{theorem:nonautonomous staffans weiss} are verified. Moreover, condition (ii) is also satisfied and then the family $\{\calA(t), t\in [0,T]  \}$ has maximal $L^p$-regularity. So, Theorem \ref{theorem:nonautonomous staffans weiss} asserts the maximal regularity of the non-autonomous version associated to our example considered in Section 4.3.1.
\end{remark}

\medskip

\medskip

\end{document}